\documentclass[11pt,letterpaper,oneside,english]{article}
\usepackage{amssymb,amsbsy,latexsym}
\usepackage{amsmath}
\usepackage{graphics, subfigure, float}
\usepackage{fp, calc}
\usepackage{bm}
\usepackage{hyperref}

\usepackage[latin1]{inputenc}
\usepackage[american]{babel}
\usepackage[T1]{fontenc} % Important: Use before loading the fonts
\usepackage{amscd,amsthm}

\usepackage{verbatim, comment}
\usepackage[dvips]{graphicx,epsfig,color}
\usepackage{pst-all}
\usepackage{pstricks-add}
\newpsobject{showgrid}{psgrid}{subgriddiv=1,griddots=10,gridlabels=6pt}

\newtheoremstyle{theorem}{1em}{1em}{\slshape}{0pt}{\bfseries}{.}{ }{}
\theoremstyle{theorem}
\newtheorem{theorem}{Theorem}
\newtheorem*{theorem*}{Theorem}
\newtheorem{corollary}[theorem]{Corollary}

\newtheorem{lemma}[theorem]{Lemma}

\newtheorem*{claim*}{Claim}
\newtheorem{conjecture}[theorem]{Conjecture}
\newtheorem*{conjecture*}{Conjecture}
\newtheorem*{problem*}{Problem}
\newtheorem{definition}{Definition}
\newtheorem*{definition*}{Definition}

\theoremstyle{remark}

\newtheorem*{remark*}{Remark}

\newtheorem*{algorithm*}{Algorithm}

\newenvironment{proofofclaim}{\vspace{1ex}\noindent{\emph{Proof of claim.}}\hspace{0.5em}}
   	    {\hfill$\lozenge$\vspace{1ex}}

\providecommand{\setN}{\mathbb{N}}
\providecommand{\setZ}{\mathbb{Z}}
\providecommand{\setQ}{\mathbb{Q}}
\providecommand{\setR}{\mathbb{R}}

\newcommand{\vol}{\textrm{vol}} 
\newcommand{\conv}{\textrm{conv}}

\newcommand{\rk}{\textrm{rk}} 
\newcommand{\CFenc}{\mathbf{CF}^{\textrm{enc}}} 
\newcommand{\CF}{\mathbf{CF}} 
\newcommand{\proj}{\textrm{proj}} 
\newcommand{\xc}{\textrm{xc}} 

\floatstyle{ruled}
\newfloat{algorithm}{tbp}{loa}
\floatname{algorithm}{Algorithm}
\newfloat{algorithm*}{tbp}{loa}
\floatname{algorithm*}{Algorithm}

\usepackage[displaymath,textmath,sections,graphics, subfigure, floats]{preview} 
\PreviewEnvironment{enumerate} % Attention: Do not use comma-seperated lists here 
\PreviewEnvironment{lemma} 
\PreviewEnvironment{myproblem}
\PreviewEnvironment{corollary}
\PreviewEnvironment{abstract} 
\PreviewEnvironment{defn} 
\PreviewEnvironment{center} 
\PreviewEnvironment{pspicture} 
\def\placeIIID#1#2#3#4{
           \FPeval{\x}{(#1) + 0.3*(#2)} 
           \FPeval{\y}{0.5*(#2) + (#3)} 
           \rput[c](\x,\y){#4}
}

\def\placeProj#1#2#3#4#5{
           \FPeval{\x}{(#1) + 0.3*(#2)} 
           \FPeval{\y}{0.5*(#2) + (#3)} 
           \pnode(\x,\y){#4}
           \FPeval{\x}{(#1) + 0.3*(#2)} 
           \FPeval{\y}{0.5*(#2)} 
           \pnode(\x,\y){#5}
}

\psset{arrowsize=6pt, labelsep=2pt, linewidth=1.5pt}
\makeatother

\begin{document}

\title{Some $0/1$ polytopes need exponential size extended formulations}
%On the size of Extended Formulations}

\author{Thomas Rothvoß\thanks{Supported by the Alexander von Humboldt Foundation within the Feodor Lynen program.}\vspace{3mm}\\M.I.T. \\
{ \tt{rothvoss@math.mit.edu}} }

\maketitle

\begin{abstract}
\noindent We prove that there are $0/1$ polytopes $P \subseteq \setR^n$ that do not
admit a compact LP formulation. More precisely we show that for every $n$ 
there is a sets $X \subseteq \{ 0,1\}^n$
%such that the slack-matrix of $P = \conv(X)$ must have a
such that $\conv(X)$ must have extension complexity 
at least $2^{n/2\cdot(1-o(1))}$. In other words, every polyhedron $Q$ that can 
be linearly projected on $\conv(X)$ must have exponentially many facets.
%Moreover, there exists a family of matroids such that ever

In fact, the same result also applies if $\conv(X)$ is restricted to be a matroid polytope. 

Conditioning on $\mathbf{NP} \not\subseteq \mathbf{P_{/poly}}$, our result
rules out the existence of any compact formulation for the TSP polytope, 
even if the formulation may contain arbitrary real numbers. 
\end{abstract}

\section{Introduction}

Combinatorial optimization deals with finding the best solution out of a finite number of choices $X \subseteq \{0,1\}^n$, 
e.g. finding the cheapest spanning tree in a graph. If possible one aims of course to design a polynomial time algorithm. 
However another popular way to study combinatorial problems is to express the convex hull $P = \conv(X)$ by linear
inequalities $Ax \leq b$, i.e. describing them as the solutions of a linear program. A drawback of this approach is
that in general an exponential number of inequalities is needed. In principle one could use the Ellipsoid 
method to optimize these systems, if at least the separation problem can be solved in polynomial time. But in practice this method is 
considered to be not applicable. A more satisfactory approach is to allow polynomially many extra variables 
in order to reduce the number of necessary inequalities to a polynomial. This is called a
\emph{compact formulation} $P = \{ x \mid \exists y : Ax + Uy \leq b\}$. Such compact formulations exist for example for 
the spanning tree polytope~\cite{CF4SpanningTree-KippMartin1991}, the parity polytope and 
the permutahedron (see \cite{CombinatorialOptimizationABC-Schrijver} for an extensive account).

The advantages of such a compact formulation are that (1) one can now optimize any linear function 
over $X$ in polynomial time; 
(2) one can solve the problem with a powerful general purpose LP solver, without the need to implement
a custom-tailored algorithm. %In practice many 
% and (3) in case the considered problem appears as a subproblem of another $\mathbf{NP}$-hard
%problem, the latter one can sti integer linear
%program, the

This naturally leads to the question for which problems such a compact formulation does \emph{not} exist.
Yannakakis~\cite{ExpressingCOproblemsAsLPs-Yannakakis1991} showed that 
the TSP polytope $P_{\textrm{TSP}}$ (the convex full of the characteristic vectors of all Hamiltonian
cycles in the complete graph on $n$ nodes) does not have a subexponential size 
\emph{symmetric} formulation. Surprisingly the same result holds true for the matching
polytope, though here a complete description of all facets is known due to Edmonds~\cite{matchingPolytop-Edmonds1965} 
and the problem itself as well as the separation problem are solvable in polynomial time. 
Kaibel, Pashkovich
and Theis~\cite{SymmetryMatters-KaibelPashkovichTheis-IPCO2010} demonstrate that 
symmetric formulations are in some cases more restricted by  proving that there is a compact non-symmetric 
formulation for all $\log n$-size matchings, while symmetric formulations still need size $n^{\Omega(\log n)}$.

%In most cases, the cost functions are linear and the set of solutions is of
%the form $X \subseteq \{0,1\}^n$. Hence, one natural approach is to express the convex full 
%of the solutions $X$ (denoted by $P = \conv(X)$) using linear inequalities.  
%But in most cases the polytope $P$ has exponentially many facets. THis problem
%adding extra variables, allows a compact formulation. 
However, it remains a fundamental open problem to show that the matching polytope or 
the TSP polytope do not admit any  non-symmetric compact formulation. 
In fact, it was even an open problem to prove that there 
\emph{exists} any family of $0/1$ polytopes without a compact formulation\footnote{This was posed as
an open problem by Volker Kaibel on the 1st Carg{è}se Workshop in Combinatorial Optimization.}. 
In this paper we answer this question affirmatively.

Our idea is based on a counting argument similar to Shannon's 
theorem~\cite{CircuitSizesShannon1949} (see also \cite{ComputationalComplexity-AroraBarak2009}) for lower bounds on
circuit sizes: Let us assume for the sake of contradiction that 
all $n$-dimensional 0/1 polytopes have a
compact formulation $P = \{ x \mid \exists y \geq \mathbf{0}: Ax + Uy = b \}$
of polynomial size  $r(n)$.
Since there are doubly-exponentially many 0/1 polytopes, there must also be 
at least that many formulations of size $r(n)$. This would lead to a contradiction
under the additional
assumption that all coefficients in the system $Ax + Uy = b$ have polynomial encoding length. 
Unfortunately there is no known result which guarantees that the coefficients of $U$ will even be rational 
% the description of the extension  $P = \{ x \mid Ax + Uy \leq b\}$ might be irrational 
and already a single real number can contain an infinite amount of information\footnote{Note that
the usual argument that a polytope with rational vertices admits rational inequalities and vice versa does 
not apply, since both, the vertices and the inequalities of the extension polyhedron might be irrational.} ruling out a simple counting 
argument.
%Note that allowing arbitrary numbers can in general have very strange effects: For example the $\mathbf{NP}$-hard ${\textsc{Partition}}$ problem can be solved within $O(n^4)$ steps on a Real-RAM~\cite{ComputationalComplexity-AroraBarak2009} -- 
% just that the intermediate numbers in the computation need an exponential encoding length.
 
\subsection*{Our contribution}

In our approach, we bypass these difficulties by selecting a linearly independent subsystem of $Ax + Uy = b$ which maximizes the volume of the spanned parallelepiped; 
then we discretize the entries of $U$. We thus obtain a subsystem  $\bar{A}x + \bar{U}y = \bar{b}$
with the property that $x \in X$ if and only if there
is a short certificate $y$ such that $\bar{A}x + \bar{U}y \approx \bar{b}$ for the rounded system.
Secondly, all numbers in $\bar{A},\bar{U},\bar{b}$ have an encoding length 
which is bounded by a polynomial in $n$. 
In other words, this construction defines an injective map, taking a set $X$ as
input and providing  $(\bar{A},\bar{U},\bar{b})$. %This construction maps 
Since there are doubly-exponentially many sets $X \subseteq \{ 0,1\}^n$ and by injectivity, 
the number of such systems $(\bar{A},\bar{U},\bar{b})$ must also
be doubly-exponential, which then implies the result.

It is folklore, that if $\mathbf{NP}$ problems do not all have polynomial size 
circuits, then the TSP polytope does not admit a compact formulation in which 
the numbers are rationals with polynomial encoding length. We can argue that
the latter condition can be omitted.

\section{Related work}

A formulation of size $O(n \log n)$ for the permutahedron was provided by Goemans~\cite{PermutahedronNlogNformulation_Goemans10}. 
In fact, \cite{PermutahedronNlogNformulation_Goemans10} also showed that this is tight up to constant factors. 
The lower bound of~\cite{PermutahedronNlogNformulation_Goemans10} is based on the insight that the number of facets of any extension must be at least logarithmic in the
number of vertices of the target polytope (which is $n!$ for the permutahedron). 
The perfect matching polytope for planar graphs and graphs with bounded genus does admit a compact formulation~\cite{CompactForm4MatchingPolytopeInPlanarGraphs-Baharona1993, CompactLPforPerfectMatchingOnBoundedGenusSurfacesGerards1991}.
A useful tool to design such formulations is the Theorem of Balas~\cite{DisjunctiveProgramming-Balas85,DisjunctiveProgramming-Balas98}, which describes the convex hull of the
union of polyhedra. 
For $\mathbf{NP}$-hard problems, one can of course not expect the existence of any \emph{exact} compact formulation.
Nevertheless, Bienstock~\cite{ApproximateKnapsackFormulation-Bienstock2008} gave an approximate formulation of size $n^{O(1/\varepsilon)}$ for the Knapsack polytope. 
This means, optimizing any linear function over the approximate polytope will give the optimum Knapsack value, up
to a $1+\varepsilon$ factor.  
For a more detailed literature review, we refer to the surveys of Conforti, Cornu{é}jols and Zambelli~\cite{ExtendedFormulationsSurvey-ConfortiCornuejolsZambelli-2010} and of Kaibel~\cite{ExtendedFormulationsInCombOptKaibel2011arXiv}.
%says that given $k$ allows

%However, this method could not even give a
%superlinear bound lower for $0/1$ polytopes. 

%(more precisely if $\mathbf{NP} \subsetneq \mathbf{P_{/poly}}$ then $P_{\textrm{TSP}}$ does not have a poly-size LP with 
%coefficients with polynomial size). 

\section{Preliminaries}

Let $P \subseteq \setR^n$ be a polytope with non-redundant inequality representation  $P  = \{ x \in \setR^n \mid Ax \leq b\}$.
An \emph{extension} is a polyhedron $Q \subseteq \setR^m$ together with a linear
projection $p : \setR^m \to \setR^n$ such that $p(Q) = P$. An \emph{extended formulation} is a description
of $Q$ with linear inequalities and equations $Q = \{ z \in \setR^m \mid Cz \leq c, \; Dz = d \}$ (together with $p$). 
The \emph{size} of the extended formulation
is the number of inequalities in the description, i.e. the number of rows in $C$.
We do not need to account for the number of equations, since they can always be eliminated.
Now we can define the \emph{extension complexity} $\xc(P)$ as the smallest size of any extended formulation
(see \cite{ExtendedFormulationsInCombOptKaibel2011arXiv} for more details).

Let $X = \{ x_1,\ldots,x_v \} \subseteq P$ be the \emph{vertices} (or \emph{extreme points}) of $P$ 
% of binary vectors, we define $P = \conv(X)$ as the spanned polytope.
and let $f$ be the number of inequalities in the description $P  = \{ x \in \setR^n \mid Ax \leq b\}$. % be a 
%non-redundant inequality representation with $f$ constraints.
%Let $v = |X| \approx {2^n}$ be the number of vertices of $P$ and $f$ be the number of facets. 
%. It easily follows from Cramers rule
%(for example by extending $w_1,\ldots,w_k$ with an arbitrary basis for
% the orthogonal complement of $span(w_1,\ldots,w_k)$) that
Then the \emph{slack-matrix} $S\in \setR^{f× v}$ of $P$ is defined 
by $S_{ij} = b_i - A_i x_j$. Recall that the \emph{rank} of a matrix $S$ is the smallest
$r$ such that one can factor $S = UV$, where $U$ is a matrix with $r$ columns and $V$ 
is a matrix with $r$ rows. 
A notion which is very important for studying extended formulations is the \emph{non-negative rank} of a matrix:
\[
 \rk_+(S) = \min\{ r \mid \exists U \in \setR_{\geq 0}^{f× r}, V \in \setR_{\geq 0}^{r× v}: S = UV \}
\]
Note that given a matrix $A \subseteq \setQ_{\geq }^{m × n}$, deciding whether $\rk(A) = \rk_+(A)$ is $\mathbf{NP}$-hard~\cite{ComplexityOfNonnegativeMatrixFactorizationVavasis09}.
%Hence we may
%define $\rk_+(P)$ as the non-negative rank for any of its slack matrices and
%$\rk_+(X) := \rk_+(\conv(X))$.  
%An \emph{extension} of $P$ is a polytope
%\[
% Q = \{ (x,y) \mid Rx + Uy = c, y \geq \mathbf{0}\} 
%\]
%such that $P$ is the \emph{linear projection} on its $x$ variables, i.e.
%\[
% P = \{ x \mid \exists y \geq \mathbf{0}: Rx + Uy = c \}.
%\] 
A basic theorem concerning extended formulations, is the insight of Yannakakis, 
that the non-negative factorization of the slack-matrix with minimum $r$ gives the smallest extension:
\begin{theorem}[Yannakakis~\cite{ExpressingCOproblemsAsLPs-Yannakakis1991}] \label{thm:Yannakakis}
Let $P$ be a polytope with vertices $X = \{ x_1,\ldots,x_v\}$, non-redundant inequality
description $P=\{ x \in \setR^n \mid Ax \leq b\}$ and corresponding slack matrix $S$.
Then  $\xc(P) = \rk_+(S)$. 
%\[
%\rk_+(P) = \min \{ r \mid \exists \textrm{ extension } Q=\{ (x,y)\in\setR^{n+ r} \mid Rx+Uy=c, y\geq \mathbf{0}\} \textrm{ of } P\}
%\]
%\item[(II)] 
Moreover, for any factorization $S = UV$ with $U,V \geq \mathbf{0}$ one can write
 $P = \{ x \in \setR^n \mid \exists y \geq \mathbf{0}: Ax + Uy = b \}$ and for every $x_j \in X$ one has $Ax_j + U\cdot V^j = b$.
\end{theorem} 
In other words: Given a polytope $P = \{ x \in \setR^n \mid Ax \leq b\}$, the smallest extension can be found by
factoring the slack matrix $S$ into non-negative factors $U$ and $V$ with minimum number of columns/rows. 
Then the smallest extended formulation comprises of 
$Q = \{ (x,y) \in \setR^{n} × \setR^{\xc(P)} \mid Ax + Uy = b, \; y \geq \mathbf{0}\}$ together with the 
\emph{projection on the $x$-variables} $\proj_x(Q) = \{ x \in \setR^n \mid \exists y: (x,y) \in Q\}$.
While for a polytope $P$, the inequality description $Ax \leq b$ is not unique,
Theorem~\ref{thm:Yannakakis} implies that the non-negative rank is the 
same for all these descriptions. 

% and the certificate
%for the $j$th vertex $x_j$ to be in t

% On the first view, extending a system $Ax \leq b$ to $Ax + Uy = b, y \geq \mathbf{0}$ does not seem to be very helpful, since the number of constraints remains the same. 
%But one can take a  linearly independent subsystem of at most $n + \rk_+(P)$ many equations and write it in inequality form $P = \{ x \mid A'x + B'y \leq b'\}$ with at most  $3n + 2\rk_+(P)$ many inequalities.
%Also the reverse is true, i.e. given any extension $P = \{ x \in \setR^n \mid Ax + By \leq b\}$ in inequality form with $B \in \setR^{m × r}$ then $\rk_{+}(P) \leq m+r$.
%In other words, up to constant factors all notions of compact formulations are equivalent we may indeed focus on the representation from Theorem~\ref{thm:Yannakakis}.

For any matrix $A$, we denote its $i$th row by $A_i$ and the
$i$th column by $A^i$.
For linearly independent vectors $w_1,\ldots,w_k \in \setR^n$, we define 
$\vol(w_1,\ldots,w_k)$ as the $k$-dimensional volume of the parallelepiped, 
spanned by $w_1,\ldots,w_k$. Hence for $k=n$ one has 
$\vol(w_1,\ldots,w_k) = |\det(B)|$ where $B$ is a matrix, having $w_1,\ldots,w_k$
as column vectors in an arbitrary order. Note that for any vector $w \in span(w_1,\ldots,w_k)$,
there are unique coefficients $\lambda \in \setR^k$ such that
$w = \sum_{i=1}^k \lambda_iw_i$ and by Cramer's rule 
\[
  |\lambda_i| =  \frac{\vol(w_1,\ldots,w_{i-1},w,w_{i+1},\ldots,w_k)}{\vol(w_1,\ldots,w_k)}.
\]
For $q \in \setR$, let $q\setZ_{\geq 0} = \{ 0,q,2q,\ldots\}$ denote all non-negative integer multiples of $q$.

\section{A lower bound for general 0/1 polytopes}

%First observe that zero columns of $U$ (rows of $V$) are redundant and
%could just be deleted (hence lowering $\textrm{rk}_+(S)$. 
In the following we fix a set  $X\subseteq \{ 0,1\}^n$.  % and an inequality description $P = \{ x \in \setR^n \mid Ax \leq b \}$ of  $\conv(X)$.
It is well known, that one can choose a matrix $A$ and a vector $b$ with integral entries such that
$P = \{ x \in \setR^n \mid Ax \leq b \} = \conv(X)$, while the
absolute values of any entry in $A$ and $b$ are bounded by $\Delta := \Delta(n) := (\sqrt{n+1})^{n+1} \leq 2^{n \log(2n)}$~(see e.g. Cor.~26 in \cite{LecturesOn0-1Polytopes-Ziegler}). % For $n\geq 3$, we can use the estimate  $\Delta \leq 2^{n \log n}$.
Let $S$ be the corresponding slack-matrix, then $S$ is non-negative by definition and integral, 
since $A,b$ and all vertices are integral. More precisely  $S_{ij} = b_j - A_ix_j \in \{ 0,\ldots, (n+1)\Delta\}$. 
%Let $r := \xc(P)$ be its non-negative rank
Let $S=UV$ be any non-negative factorization, i.e. $U \in \setR_{\geq 0}^{f × r}$ and $V \in \setR_{\geq 0}^{r × v}$.
As already argued above, we cannot make any assumption on the rationality/encoding length of the coefficients of $U$ and $V$. 
But what we can do is to bound their absolute values. 

Observe that if we simultaneously scale a column $\ell$ of $U$ by $\lambda>0$ and row $\ell$ of $V$ by $\frac{1}{\lambda}$,
then the matrix product $UV$ stays invariant. % since
%\[
% (U\cdot V)_{ij} = \sum_{k=1}^r U_{ik}\cdot V_{jk} = \lambda U_{i\ell}\cdot \frac{1}{\lambda}V_{j\ell} + \sum_{k\neq \ell} U_{ik}\cdot V_{jk}
%\]
Thus we may scale the rows and columns such that $\| U^{\ell}\|_{\infty} = \|V_{\ell}\|_{\infty}$ (if $U^{\ell} = \mathbf{0}$, then
we can just set $V_{\ell} := \mathbf{0}$ as well). We call such pairs of matrices 
\emph{normalized}.

\begin{lemma} \label{lem:NormOfNormalizedMatrices}
For normalized matrices, one has $\|U\|_{\infty} \leq \Delta$ and $\|V\|_{\infty} \leq \Delta$.
\end{lemma}
\begin{proof}
Assume for the sake of contradiction that %For symmetry reasons, we may assume that
 $U_{i\ell} > \Delta$. Thus $\|V_{\ell}\|_{\infty} > \Delta$, hence
there must be an entry $V_{\ell j} > \Delta$. Then $S_{ij} = U_i\cdot V^j \geq U_{i\ell}\cdot V_{\ell j} > \Delta^2 \geq (n+1)\Delta$,
which is a contradiction. 
%But this is a contradiction since the slackmatrix has only entries 
%from $0,\ldots,(n+1)\Delta$.
\end{proof}
%Unfortunately, the entries of $U$ or $V$ might be arbitrary real numbers in $[0,\sqrt{\Delta}]$.

Recalling Theorem~\ref{thm:Yannakakis}, we can write
 $\conv(X) = \{ x \in \setR^n \mid \exists y \in \setR^{\xc(\conv(X))}_{\geq 0} : Ax + Uy = b\}$. Our main technical ingredient is to select 
a linear independent subsystem  $\bar{A}x + \bar{U}y = \bar{b}$ of $Ax + Uy = b$
such that the entries of $\bar{U}$ can be rounded to rational numbers with small encoding length and still 
$x \in X$ iff $\bar{A}x + \bar{U}y \approx \bar{b}$ for some $y$.

\begin{theorem} \label{thm:DescriptionOfX}
For any non-empty $X \subseteq \{ 0,1\}^n$, 
there are matrices $\bar{A} \in \setZ^{(n+r) × n}, \bar{U} \in (\frac{1}{4r(n+r)\Delta}\setZ_{\geq 0})^{(n+r) × r}$
and a vector $\bar{b} \in \setZ^{n+r}$ with $\|\bar{A}\|_{\infty}, \|\bar{b}\|_{\infty}, \|\bar{U}\|_{\infty} \leq \Delta$
such that 
\[
  X = \bigg\{ x \in \{ 0,1\}^n \mid \exists y\in [0,\Delta]^r: \| \bar{A}x + \bar{U}y - \bar{b} \|_{\infty} \leq \frac{1}{4(n+r)} \bigg\}
\]
Here is  $r := \xc(\conv(X))$ and $\Delta := \Delta(n) := (\sqrt{n+1})^{n+1}$.
\end{theorem}
\begin{proof}
%The claim is trivially true if $X$ is empty, hence assume
Let $X = \{ x_1,\ldots,x_{v} \}$ and
let $Ax \leq b$ with $A \in \setZ^{f × n}$ and $b \in \setZ^{f}$ be a non-redundant description of $\conv(X)$ 
with $\| A\|_{\infty}, \| b\|_{\infty} \leq \Delta$. Furthermore let $S \in \setZ_{\geq 0}^{f × |X|}$ be the corresponding slack matrix. 
%We assume that $U$ and $V$ are normalized non-negative factors of the slackmatrix of $P = \conv(X)$.
By Yannakakis' Theorem~\ref{thm:Yannakakis}, we can 
write $P = \conv(X) = \{ x \in \setR^n \mid \exists y \in \setR^{r}: Ax + Uy = b, \; y \geq \mathbf{0} \}$ where $U,V$
are the non-negative factorization of the slack-matrix, i.e. $S = UV$. 
%Recall that $A$ and $b$ have integral entries of absolute value at most $\Delta$.
By Lemma~\ref{lem:NormOfNormalizedMatrices} we may assume that 
$\| U\|_{\infty}, \| V\|_{\infty} \leq \Delta$.
Let $W = \textrm{span}(\{(A_i,U_i) \mid i=1,\ldots,f\})$ be the span of the
constraint matrix of the system $Ax + Uy = b$ and let $k = \dim(W)$ be its dimension.
%Let $W^{\perp} \subseteq \setR^{n+r}$ be the orthogonal complement (possibly $W^{\perp} = \{ \mathbf{0}\}$)  and let $k = \dim(W)$. Let $w_1,\ldots,w_{k}$ be any basis of $W^{\perp}$.
Choose $I\subseteq\{ 1,\ldots,f\}$ of size $|I| = k$ such that
$\vol(\{ (A_i,U_i) \mid i \in I\})$ is maximized. 
Recall that $U_I$ is the matrix $U$, restricted to the rows in $I$. 
%Observe that $\{ (A_i,U_i) \mid i \in I\} \cup \{ w_1,\ldots,w_k\}$ form a basis of $\setR^{n + r}$.
Let $U_I'$ be the matrix $U_I$ where coefficients are rounded down to the nearest 
multiple of $\frac{1}{4r(n+r)\Delta}$. Our choice will be $\bar{A} := A_I,\bar{U}:=U_{I}', \bar{b}:=b_I$, hence it remains to show that
\[
  X \stackrel{!}{=}  \bigg\{ x \in \{ 0,1\}^n \mid \exists y\in [0,\Delta]^r: \| A_Ix + U_I'y - b_I \|_{\infty} \leq \frac{1}{4(n+r)} \bigg\} =: Y
\]
%$\underline{\textrm{Claim ``}\subseteq\textrm{'':}}$ 
\begin{claim*} $X\subseteq Y$. \end{claim*} % \renewcommand{\qedsymbol}{$\lozenge$}
\begin{proofofclaim}
Consider a vector $x_j \in X$. Using Yannakakis' Theorem~\ref{thm:Yannakakis},
we can simply choose $y := V^j \geq \mathbf{0}$ and have
$Ax_j + U\cdot y = b$. 
Due to normalization,   $\| y \|_{\infty} \leq \|V\|_{\infty} \leq \Delta$.
Note that $\|U - U'\|_{\infty} \leq \frac{1}{4r(n+r)\Delta}$. By the triangle inequality
\begin{eqnarray*}
\|A_Ix_j+U_I'y - b_I\|_{\infty} &\leq& %\|Ax_j + U'y - b\|_{\infty} = 
\|\underbrace{A_Ix_j + U_Iy - b_I}_{=\mathbf{0}} + (U_I'- U_I)y\|_{\infty} \\
&\leq& r\cdot \underbrace{\|U_I'-U_I\|_{\infty}}_{\leq \frac{1}{4r(n+r)\Delta}} \cdot \underbrace{\|y\|_{\infty}}_{\leq \Delta} \leq \frac{1}{4(n+r)}
\end{eqnarray*}
Thus $x_j \in Y$.
\end{proofofclaim}
%$\underline{\textrm{Claim ``}\supseteq\textrm{'':}}$ 
\begin{claim*} $X\supseteq Y$. \end{claim*}
\begin{proofofclaim}
We show that for $x \in \{ 0,1\}^n$ with $x \notin X$ one has $x \notin Y$. 
Since $x \notin X$, there must be a row $\ell$ with
$A_{\ell}x > b_{\ell}$. Since $A,b$ and $x$ are integral, one even has $A_{\ell}x \geq b_{\ell}+1$.
Unfortunately $\ell$ is in general not among the selected constraints $I$.
But there are unique coefficients $\lambda \in \setR^{k}$ such that we can express
constraint $A_{\ell}x + U_{\ell}y = b_{\ell}$ as a linear combination of those in $I$, i.e.
\[
  \begin{pmatrix} A_{\ell}, U_{\ell} \end{pmatrix} = \sum_{i\in I} \lambda_{i}\begin{pmatrix} A_{i}, U_{i} \end{pmatrix}.
\]
Note that automatically we have $\sum_{i\in I} \lambda_ib_i = b_{\ell}$, since
otherwise the system $Ax + Uy = b$ could not have any solution $(x,y)$ at all and $X=\emptyset$.
The next step is to bound the coefficients $\lambda_i$. 
Here we recall that by Cramer's rule
\[
%  |\lambda_{i}| = \dfrac{\vol\bigg(\bigg\{ \begin{pmatrix} A_{i'} \\ U_{i'} \end{pmatrix} \mid i' \in I \backslash \{ i \} \cup \{ \ell \} \bigg\}\bigg)}{\vol\bigg(\bigg\{ \begin{pmatrix} A_{i'} \\ U_{i'}\end{pmatrix}  \mid i' \in I\bigg\}\bigg)} \leq 1
  |\lambda_{i}| = \dfrac{\vol\big(\big\{ ( A_{i'}, U_{i'}) \mid i' \in I \backslash \{ i \} \cup \{ \ell \} \big\}\big)}{\vol\big(\big\{ ( A_{i'}, U_{i'})  \mid i' \in I\big\}\big)} \leq 1
\]
since we picked $I$ such that $\vol(\{ ( A_{i'}, U_{i'})  \mid i' \in I\})$ is maximized.
Fix an arbitrary $y \in [0,\Delta]^r$, then
\begin{eqnarray}
 1 \leq  |\underbrace{A_{\ell}x - b_{\ell}}_{\geq 1} + \underbrace{U_{\ell}y}_{\geq 0} | % \\
&=& \Big| \sum_{i\in I} \lambda_{i} (A_{i}x  - b_i + U_iy)\Big|  \label{eq:MainProof} \\
&\leq& \sum_{i\in I} \underbrace{|\lambda_{i}|}_{\leq 1}\cdot |A_{i}x  - b_i + U_iy| \nonumber \\
% &\stackrel{|I| \leq n+r}{\leq}&
&\leq& (n+r)\cdot \|A_Ix - b_I + U_Iy\|_{\infty} \nonumber
\end{eqnarray}
using the triangle inequality and the fact that $|I| \leq n+r$.
Again making use of the triangle inequality yields
\begin{eqnarray}
 \|A_Ix - b_I + U_Iy\|_{\infty}  
&=&  \|A_Ix - b_I + U_I'y + (U_I-U_I')y\|_{\infty} \label{eq:MainProofII} \\
&\leq&  \|A_Ix - b_I + U_I'y \|_{\infty} + r\cdot \underbrace{\| U_I - U_{I}' \|_{\infty}}_{\leq \frac{1}{4r(n+r)\Delta}} \cdot \underbrace{\|y\|_{\infty}}_{\leq \Delta} \nonumber \\
&\leq& \|A_Ix - b_I + U_I'y \|_{\infty} + \frac{1}{4(n+r)} \nonumber
\end{eqnarray}
Combining \eqref{eq:MainProof} and \eqref{eq:MainProofII} gives  $\| A_Ix - b_I + U_I'y\|_{\infty} \geq \frac{1}{n+r} - \frac{1}{4(n+r)} \geq \frac{1}{2(n+r)}$  and consequently $x \notin Y$.
%\begin{eqnarray}
% 1 &\leq & |\underbrace{A_{\ell}x - b_{\ell}}_{\geq 1} + \underbrace{U_{\ell}y}_{\geq 0} | \label{eq:MainProof} \\
%&=& \Big| \sum_{i\in I} \lambda_{i} (A_{i}x  - b_i + U_iy)\Big|  \nonumber \\
%&\leq& \sum_{i\in I} \underbrace{|\lambda_{i}|}_{\leq 1}\cdot |A_{i}x  - b_i + U_iy| \nonumber \\
%&\stackrel{|I| \leq n+r}{\leq}& (n+r)\cdot \|A_Ix - b_I + U_Iy\|_{\infty} \nonumber \\
%&=& (n+r)\cdot \|A_Ix - b_I + U_I'y + (U_I-U_I')y\|_{\infty} \nonumber \\
%&\leq&  (n+r)\cdot\|A_Ix - b_I + U_I'y \|_{\infty} + \underbrace{(n+r)\cdot r\cdot \underbrace{\| U_I - U_{I}' \|_{\infty}}_{\leq \frac{1}{4r(n+r)\Delta}} \cdot \underbrace{\|y\|_{\infty}}_{\leq \Delta}}_{\leq 1/4} \nonumber
%\end{eqnarray}
%Hence $\| A_Ix - b_I + U_I'y\|_{\infty} \geq \frac{1}{2(n+r)}$. % and the theorem is proven.
\end{proofofclaim} % \renewcommand{\qedsymbol}{$\square$}

The assertion of the Theorem follows. 
Note that by padding empty rows, we can ensure that $\bar{A},\bar{U},\bar{b}$ have exactly $n+r$ rows.
\end{proof}

\begin{theorem} \label{thm:NonnegativeRankLowerBound}
For any $n\in \setN$, there exists a set $X \subseteq \{ 0,1\}^n$ such that 
$\xc(\conv(X)) \geq \Omega(2^{n/2} / \sqrt{n \log (2n)})$. %$rk_+(X) \geq 2^{n/2\cdot (1-o(1)}$.
%In other words, for any $n$, there exists a $0/1$ polytope $P=\conv(X)$ for which any
%extension must have either $2^{\frac{n}{2}(1-o(1))}$ many facets or variables. 
\end{theorem}
\begin{proof}
Let $R := R(n)$ be the maximum value of $\xc(\conv(X))$ over all $X \subseteq \{ 0,1\}^n$.
%Suppose for the
%sake of contradictions, that all sets $X$ have a non-negative rank at
%most $R:=R(n)$. %In fact, by adding zero rows and columns to the 
%factors $U,V$ of the slack-matrix
In the following, we use that $R \leq 2^n$ (otherwise, there is nothing to show).
The construction in Theorem~\ref{thm:DescriptionOfX} implicitly defines a function $\Phi$ which
maps a set $X$ to a system $(\bar{A},\bar{U},\bar{b})$\footnote{The initial system $Ax \leq b$ describing $\conv(X)$ might not be unique, as well as 
index set $I$. For $\Phi$ to be well defined one can make an arbitrary canonical choice, like  choosing $Ax\leq b$ and $I$ lexicographical minimal.}. The important observation is that due to Theorem~\ref{thm:DescriptionOfX}, 
for a given system  $(\bar{A},\bar{U},\bar{b})$, one can reconstruct the 
corresponding set $X$. In other words, the function $\Phi$ is injective. 
 In fact, adding zero rows and columns
to those matrices does not change the claim, hence we may assume that
$\bar{A}$ is an $(n + R) × n$ matrix and $\bar{U}$ is an $(n+R) × R$ matrix.
%$\Psi : \mathcal{P} \bar{A} \to  \in \setR^{R }$ 
%. Since this system 
%uniquely defines $X$, the map $\phi$ must be injective.
Every entry in $\bar{U}$ has absolute value at most $\Delta$ and is a multiple 
of $\frac{1}{4r(n+r)\Delta}$ for some $r \in \{ 1,\ldots,R\}$. In other words, the domain for each entry contains
at most $\sum_{r=1}^R 2\cdot4r(n+r)\Delta\cdot \Delta  \leq 8R^2(n+R^2)\Delta \leq 16\Delta^5$ many possible values (here we use the
generous estimates $R \leq 2^n \leq \Delta$ and $n \leq \Delta$).
%For any entry there are at most $8R(n+R)\Delta^2$ many possibilities. 
%Since $X$ can be reconstructed from $(\bar{A},\bar{U},\bar{b})$, 
By injectivity of $\Phi$, the number of sets $X$ (which is $2^{2^n}-1$) cannot
be larger than the number of systems $(\bar{A}, \bar{U}, \bar{b})$. Thus
\[
  2^{2^n}-1 \leq (16\Delta^5)^{(n+R+1) \cdot (n+R)} % \leq 2^{(2Cn\log n + 5n)\cdot (3n^2 + 3R^2)}
\leq 2^{C(n^4 + n \log (2n)\cdot R^2)}
\]
for some constant $C>0$. Hence $R \geq C'\cdot 2^{n/2} / \sqrt{n \log (2n)}$ for some $C'>0$.
% and the claim follows.
%which can be rearranged to $R \geq 2^{n/2\cdot(1-o(1)}$, recalling that $\Delta = 2^{O(n\log n)}$.
\end{proof}
%We used that there are doubly-exponentially many different subsets of $\{0,1\}^n$. However, 
%for $n \geq 6$, there are also at least $2^{2^{n-2}}$ many combinatorially non-equivalent (i.e. their face lattices are 
%pairwise non-isomorphic), full-dimensional%
%$0/1$ polytopes~(see e.g. \cite{LecturesOn0-1Polytopes-Ziegler}).

%A more careful calculation should give $R \geq \Omega(2^{n/2}/\sqrt{n \log n})$.
%Alternatively:
%\begin{lemma}
%There is a function $\Delta := \Delta(n) = 2^{O(n \log n)}$ such that for any $X \subseteq \{ 0,1\}^n$
%with $r := rk(X)$
%there are $A \in \setZ^{(n + r) × n}, \bar{U} \in \setZ_{\geq 0}^{(n+r) × r}, \bar{b} \in \setZ^{n+r}$
%with $\| A \|_{\infty}, \| \bar{U} \|_{\infty}, \| \bar{b} \|_{\infty} \leq \Delta$ and
%\[
%  X = \Big\{ x \in \{ 0,1\}^n \mid \exists y\in [0,\sqrt{\Delta}]^r: \| \bar{A}x + \bar{U}y - \bar{b} \|_{\infty} \leq 1 \Big\}
%\]
%\end{lemma}

%\subsection*{Open problems}

%Of course, the above approach is not satisfactory from 

\section{A lower bound for matroid polytopes}

The main drawback of our result is that it does not rule out compact formulations for any explicitly known
polytope. However, we can extend the result to matroid polytopes. Recall that a pair $([n],\mathcal{I})$
is called a \emph{matroid} with \emph{ground set} $[n] = \{ 1,\ldots,n\}$ and \emph{independent sets} $\mathcal{I} \subseteq 2^{[n]}$, 
if (I) $I \in \mathcal{I}, J \subseteq I \Rightarrow J \in \mathcal{I}$ and (II) for all $I,J \in \mathcal{I}$
with $|I| < |J|$ there is a $z \in J \backslash I$ with $I + z \in \mathcal{I}$. 
Note that all non-trivial facet-defining inequalities for $\conv(\chi(\mathcal{I}))$ are 
of the form $\sum_{i \in S} x_i \leq r_{\mathcal{I}}(S)$ with $S \subseteq [n]$, where $r_{\mathcal{I}}$ denotes the \emph{rank function}
of the matroid ($\chi(\mathcal{I})$ denotes the set of characteristic vectors of $\mathcal{I}$). Secondly, any linear objective function can be optimized over $\conv(\chi(\mathcal{I}))$
using the greedy algorithm, which involves
calling a membership oracle a polynomial number of times. See e.g. the textbook of Schrijver~\cite{CombinatorialOptimizationABC-Schrijver}
for more details. 

Nevertheless, it is well known that the number of matroids with ground set $\{1,\ldots,n\}$ is at least
 $2^{{n \choose \lfloor n/2\rfloor}/(2n)} \geq 2^{2^n/(10n^{3/2})}$ for $n$ large enough~\cite{DoublyExponentiallyManyMatroids-Dukes2003}. 
In other words, there are doubly-exponentially many matroids. 
Using the same proof as for Theorem~\ref{thm:NonnegativeRankLowerBound} we obtain:
\begin{corollary}
There exists a family $M_n = (\{1,\ldots,n\},\mathcal{I}_n)$ of matroids such that
$\xc(\conv(\chi(\mathcal{I}_n))) = \Omega(2^{n/2} / (n^{5/4}\log(2n)))$.
%$2^{\frac{n}{2}(1-o(1))}$.
\end{corollary}

%\begin{proof}
%xxThe number of rank $r$ ($2<r<n$) paving matroids is $2^{{n \choose r}/2n}$.
%The number of different matroids with ground set $n$ is at least $2^{{n \choose \lfloor n/2\rfloor}/(2n)$
%\end{proof}

%\newpage

%\begin{quote}
%Let $X \subseteq \setZ^n$ be a finite set such that
% $\conv(X) = \{ x \in \setR^{n} \mid Ax \leq b\}$ for integral $A$ and $b$ with $\| A\|_{\infty},\| b\|_{\infty} \leq \Delta$
%Furthermore assume that the corresponding slackmatrix $S$ has $\| S\|_{\infty} \leq \Delta^2$.
%Then there are matrices $\bar{A} \in \setZ^{(n+r) × n}, \bar{U} \in (\frac{1}{4r(n+r)\Delta}\setZ_{\geq 0})^{(n+r) × r}$
%and a vector $\bar{b} \in \setZ^{n+r}$ with $\|\bar{A}\|_{\infty}, \|\bar{b}\|_{\infty}, \|\bar{U}\|_{\infty} \leq \Delta$
%such that 
%\[
%  X = \bigg\{ x \in \setZ^n \mid \exists y\in [0,\Delta]^r: \| \bar{A}x + \bar{U}y - \bar{b} \|_{\infty} \leq \frac{1}{4(n+r)} \bigg\}
%\]
%Here is  $r := \xc(\conv(X))$.
%\end{quote}

\section{Approximating $0/1$ polytopes}

%Observe that the proof of Theorem~\ref{thm:DescriptionOfX} still does not rule out that compact formulations with real numbers
%might be more powerful than those with rational numbers. 
%But it allows an arbitrary good approximation
%of $P$ as the projection of a polytope whose numbers do have a small encoding length.
In this section, we want to extend the result of Theorem~\ref{thm:DescriptionOfX} such that
any $0/1$ polytope $P$ can be arbitrarily well approximated as a projection of a 
polytope $Q$ with $O(n + \xc(P))$ facets but still small encoding length. 
See Figure \ref{thm:Approximation-P-by-Q} for an illustration. 
In the following, for any $\varepsilon > 0$, let $P + \varepsilon = \{ x+z \in \setR^n \mid x \in P, \|z\|_2 \leq \varepsilon\}$. % and $Q_x = \{ x \mid \exists y: (x,y) \in Q \}$.
\begin{theorem} \label{thm:Approximation-P-by-Q}
For any non-empty $0/1$ polytope $P = \conv(X)$ ($X \subseteq \{ 0,1\}^n$) and any $\varepsilon > 0$, 
there exists a polytope $Q = \{ (x,y) \in \setR^{n} × \setR^{\xc(P)} \mid Bx + Cy \leq d \}$ such that $B \in \setQ^{(4\xc(P)+2n)× n}, C \in \setQ^{(4\xc(P)+2n) × \xc(P)}$
and $b \in \setQ^{4\xc(P)+2n}$ % with $n+\rk_+(P)$ variables, $\rk_+(P)$ facets, 
have encoding length $poly(n,\xc(P),\log (\frac{1}{\varepsilon}))$ and $P \subseteq \proj_x(Q) \subseteq P + \varepsilon$.

Furthermore for any objective function $c \in \setR^n$, 
$\max\{ c^Tx \mid x \in \proj_x(Q)\} - \max\{ c^Tx \mid x \in P\} \leq \varepsilon \cdot \| c\|_{2}$.
\end{theorem}
\begin{proof}
Again let $P = \{ x \in \setR^n \mid Ax \leq b\}$ be a non-redundant inequality description of $P$ such that $A$ and $b$
have entries from $\{-\Delta,\ldots,\Delta\}$. Abbreviate $r := \xc(P)$. We again apply Theorem~\ref{thm:DescriptionOfX} to obtain
a system $A_I,U_I', b_I$. But this time, we round the entries in the matrix $U_I$ down 
to the nearest multiple of $\frac{\delta}{4r(n+r)\Delta}$ (instead of $\frac{1}{4r(n+r)\Delta}$), for $\delta := \min\{ \frac{1}{2 (n\Delta)^{2n+2}}, \frac{\varepsilon}{n\cdot (n\Delta)^{n}}\}$.
We choose
\[
 Q := \left\{ (x,y) \mid \| A_Ix + U_I'y- b_I\|_{\infty} \leq \frac{\delta}{4(n+r)}, y \in [0,\Delta]^r  \right\}
\]
Note that $Q$ is in fact a polytope which can be written in the form $Q = \{ (x,y) \mid Bx + Cy \leq d\}$ such that $B,C,d$ are of the claimed format. 
Furthermore the encoding length of $B,C,d$ is polynomial in $n$, $\xc(P)$
and $\log(1/\varepsilon)$\footnote{This follows from the fact that all coefficients in $B,C,d$ are 
products of $n, \xc(P), \delta, \varepsilon, \Delta$ (or their reciprocals) and $\log(\Delta) \leq O(n \cdot \log n), \log(1/\delta) \leq \log(1/\varepsilon) + O(n^2 \log n)$.}. In the remaining proof we show that  $P \subseteq \proj_x(Q) \subseteq \{ x \in \setR^n \mid Ax \leq b + \delta \mathbf{1} \} \subseteq P + \varepsilon$.

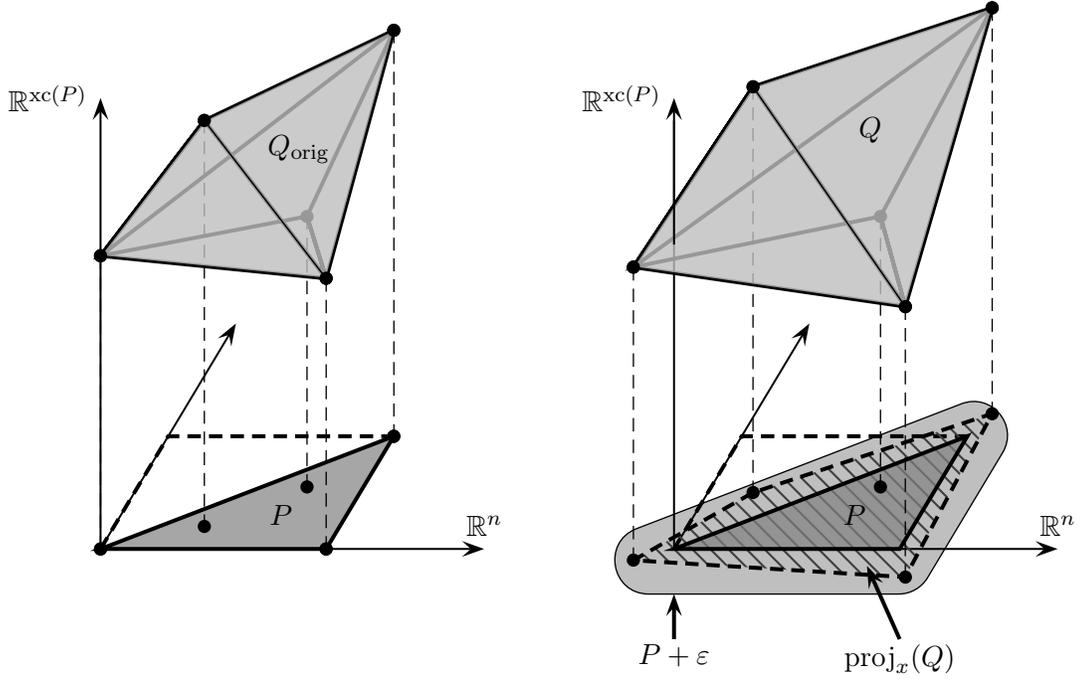
\begin{figure}
\begin{center}
\psset{unit=3.0cm}
\begin{pspicture}(0.2,-0.6)(1.5,1.8)
  \placeIIID{0}{0}{0}{\pnode(0,0){x00}}
  \placeIIID{1}{0}{0}{\pnode(0,0){x10}}
  \placeIIID{0}{1}{0}{\pnode(0,0){x01}}
  \placeIIID{1}{1}{0}{\pnode(0,0){x11}}
  \placeIIID{1.7}{0}{0}{\pnode(0,0){axesX}}
  \placeIIID{0}{2}{0}{\pnode(0,0){axesY}}
  \placeIIID{0}{0}{2}{\pnode(0,0){axesZ}}
  \ncline[linewidth=0.75pt]{->}{x00}{axesX}
  \ncline[linewidth=0.75pt]{->}{x00}{axesY}
  \ncline[linewidth=0.75pt]{->}{x00}{axesZ}
  \pspolygon[fillstyle=solid,fillcolor=gray,opacity=0.7](x00)(x10)(x11)
  \psdots(x00)(x10)(x11) \rput[c](0.8,0.15){$P$}
  \psline[linestyle=dashed](x00)(x01)(x11)

  \placeProj{0}{0}{1.3}{Q1}{Q1P}
  \placeProj{1}{0}{1.2}{Q2}{Q2P}
  \placeProj{1}{1}{1.8}{Q3}{Q3P}
  \placeProj{0.75}{0.55}{1.2}{Q4}{Q4P}
  \placeProj{0.4}{0.2}{1.8}{Q5}{Q5P}

  \psdots(Q1)(Q2)(Q3)(Q4)(Q5)(Q1P)(Q2P)(Q3P)(Q4P)(Q5P)
  \psline(Q1)(Q2)(Q3)(Q1)(Q4)(Q2)(Q4)(Q3)(Q5)(Q2)(Q5)(Q1)(Q5)
  
  \ncline[linestyle=dashed,linewidth=0.5pt]{Q1}{Q1P}
  \ncline[linestyle=dashed,linewidth=0.5pt]{Q2}{Q2P}
  \ncline[linestyle=dashed,linewidth=0.5pt]{Q3}{Q3P}
  \ncline[linestyle=dashed,linewidth=0.5pt]{Q4}{Q4P}
  \ncline[linestyle=dashed,linewidth=0.5pt]{Q5}{Q5P}
  \pspolygon[fillcolor=lightgray,fillstyle=solid,opacity=0.8,linewidth=0.5pt](Q5)(Q1)(Q2)
  \pspolygon[fillcolor=lightgray,fillstyle=solid,opacity=0.8,linewidth=0.5pt](Q5)(Q2)(Q3)

  \psdots(Q1)(Q2)(Q3)(Q5)
  \nput[labelsep=5pt]{90}{axesX}{$\setR^n$}
  \nput[labelsep=5pt]{180}{axesZ}{$\setR^{\xc(P)}$}
  \nput[labelsep=0.3]{-20}{Q5}{$Q_{\textrm{orig}}$}
\end{pspicture}
\begin{pspicture}(-1,-0.6)(1.0,1.8)
  \placeIIID{0}{0}{0}{\pnode(0,0){x00}}
  \placeIIID{1}{0}{0}{\pnode(0,0){x10}}
  \placeIIID{0}{1}{0}{\pnode(0,0){x01}}
  \placeIIID{1}{1}{0}{\pnode(0,0){x11}}
  \placeIIID{-0.8}{-0.4}{0}{\pnode(0,0){P1}}
  \placeIIID{1.2}{-0.4}{0}{\pnode(0,0){P2}}
  \placeIIID{1.2}{1.6}{0}{\pnode(0,0){P3}}
  \pspolygon[linearc=0.15,linestyle=solid,fillcolor=lightgray,fillstyle=solid,linewidth=0.5pt](P1)(P2)(P3)
  \placeIIID{1.7}{0}{0}{\pnode(0,0){axesX}}
  \placeIIID{0}{2}{0}{\pnode(0,0){axesY}}
  \placeIIID{0}{0}{2}{\pnode(0,0){axesZ}}
  \ncline[linewidth=0.75pt]{->}{x00}{axesX}
  \ncline[linewidth=0.75pt]{->}{x00}{axesY}
  \ncline[linewidth=0.75pt]{->}{x00}{axesZ}
  \psline[linestyle=dashed](x00)(x01)(x11)

  \placeProj{-0.15}{-0.1}{1.3}{Q1}{Q1P}
  \placeProj{1.1}{-0.25}{1.2}{Q2}{Q2P}
  \placeProj{1.05}{1.2}{1.8}{Q3}{Q3P}
  \placeProj{0.75}{0.55}{1.2}{Q4}{Q4P}
  \placeProj{0.2}{0.5}{1.8}{Q5}{Q5P}
  \pspolygon[linestyle=dashed,fillstyle=vlines,hatchcolor=darkgray](Q1P)(Q2P)(Q3P)(Q5P)
  \pspolygon[fillstyle=solid,fillcolor=gray,opacity=0.7](x00)(x10)(x11)
  \psdots(Q1)(Q2)(Q3)(Q4)(Q5)(Q1P)(Q2P)(Q3P)(Q4P)(Q5P)
  \psline(Q1)(Q2)(Q3)(Q1)(Q4)(Q2)(Q4)(Q3)(Q5)(Q2)(Q5)(Q1)(Q5)
  
  \ncline[linestyle=dashed,linewidth=0.5pt]{Q1}{Q1P}
  \ncline[linestyle=dashed,linewidth=0.5pt]{Q2}{Q2P}
  \ncline[linestyle=dashed,linewidth=0.5pt]{Q3}{Q3P}
  \ncline[linestyle=dashed,linewidth=0.5pt]{Q4}{Q4P}
  \ncline[linestyle=dashed,linewidth=0.5pt]{Q5}{Q5P}
  \pspolygon[fillcolor=lightgray,fillstyle=solid,opacity=0.8,linewidth=0.5pt](Q5)(Q1)(Q2)
  \pspolygon[fillcolor=lightgray,fillstyle=solid,opacity=0.8,linewidth=0.5pt](Q5)(Q2)(Q3)

  %\psdots(x00)(x10)(x11) 
  \rput[c](0.8,0.15){$P$}
  \psdots(Q1)(Q2)(Q3)(Q5)
  \nput[labelsep=5pt]{90}{axesX}{$\setR^n$}
  \nput[labelsep=5pt]{180}{axesZ}{$\setR^{\xc(P)}$}
  \nput[labelsep=0.5]{-20}{Q5}{$Q$}
  
  \pnode(0,-0.2){L3} \pnode(0,-0.4){L4} \ncline[arrowsize=6pt]{->}{L4}{L3} \nput{-90}{L4}{$P + \varepsilon$}
  \placeIIID{0.9}{-0.15}{0}{\pnode(0,0){L1}} \pnode(1,-0.4){L2} \ncline[arrowsize=6pt]{->}{L2}{L1} \nput{-90}{L2}{$\proj_x(Q)$}
  \psline[linewidth=0.75pt](0,1.35)(0,1.9)
\end{pspicture}
\caption{Visualization of Theorem~\ref{thm:Approximation-P-by-Q}\label{fig:Projection}.}
\end{center}
\end{figure}

\begin{claim*} $P \subseteq \proj_x(Q)$.\end{claim*} 
\begin{proofofclaim}
As in Theorem~\ref{thm:DescriptionOfX}, for any vertex $x_j \in P$, one has $(x_j,V^j) \in Q$
(since $\|A_Ix_j + U_{I}'V^j-b_I\|_{\infty} \leq r\cdot \|U_{I}'-U_I \|_{\infty} \cdot \|V^j\|_{\infty} \leq \frac{\delta}{4(n+r)}$).
Consequently $P \subseteq \proj_x(Q)$. 
\end{proofofclaim}

\begin{claim*} $\proj_x(Q) \subseteq \{ x \in \setR^n \mid Ax \leq b+\delta \mathbf{1} \}$.\end{claim*} 
\begin{proofofclaim}
Suppose for the sake of contradiction, that there is an $x^* \in \proj_x(Q)$ such that for some $\ell$ one has $A_{\ell}x^* > b_{\ell} + \delta$.
Revisiting again Inequalities~\eqref{eq:MainProof} and (\ref{eq:MainProofII}), we see that for any $y \in [0,\Delta]^r$ now
%\[
%  \delta \leq (n+r)\cdot \| A_Ix^* - b_I + U_I'y\|_{\infty} + \frac{\delta}{4}
%\]
\begin{eqnarray*}
 \delta &\stackrel{\eqref{eq:MainProof}}{\leq}& (n+r)\cdot \| A_Ix^* - b_I + U_I\|_{\infty} \\
&\stackrel{\eqref{eq:MainProofII}}{\leq}& (n+r)\cdot\Big( \|A_{I}x^*-b_I + U_{I}'y\|_{\infty} + r\cdot\underbrace{\| U_I-U_{I}'\|_{\infty}}_{\leq \delta / (4r(n+r)\Delta)} \cdot \underbrace{\|y\|_{\infty}}_{\leq \Delta} \Big) \\
&\leq& (n+r)\cdot \|A_{I}x^*-b_I + U_{I}'y\|_{\infty} + \frac{\delta}{4}
\end{eqnarray*}
which implies that $\| A_Ix^* - b_I + U_I'y\|_{\infty} \geq \frac{\delta}{n+r}-\frac{\delta}{4(n+r)} > \frac{\delta}{4r(n+r)}$ and consequently $x^* \notin \proj_x(Q)$. This is a contradiction. 
\end{proofofclaim}

\begin{claim*} $\{ x \in \setR^n \mid Ax \leq b + \delta \mathbf{1} \} \subseteq P + \varepsilon$. \end{claim*}
\begin{proofofclaim}
It suffices to prove that every vertex $x^*$ of $\{ x \mid Ax \leq b + \delta\mathbf{1}\}$ has a distance of
at most $\varepsilon$ to $P$.
There is a subsystem $A_Jx \leq b_J + \delta \mathbf{1}$ of $n$ constraints such that $x^*$ is the
unique solution of $A_Jx = b_J + \delta \mathbf{1}$ or in other words $x^* = A_J^{-1}(b + \delta\mathbf{1})$. 
Since $A$ has integral entries with absolute value at most $\Delta$, we know that we can write
$A_J^{-1} = (\frac{\alpha_{ij}}{\beta})_{i,j}$ with $\alpha_{ij},\beta \in \{ -(n\Delta)^{n},\ldots, (n\Delta)^{n}\}$\footnote{By Cramer's rule, every entry $(i,j)$ of the inverse of an 
$n × n$ matrix $M$ can be written as $± \frac{\det(M')}{\det(M)}$ for some submatrix $M'$ of $M$. By the Hadamard
bound, $|\det(M)| \leq \prod_{i=1}^n \| M^i \|_2 \leq (n\|M\|_{\infty})^{n}$.}. 

Let us assume for the sake of contradiction that $J$ was not a feasible basis for $P$,  i.e. 
  $A(A_J^{-1}b_J) \nleq b$. Well, then there is an index $i$ with $A_i(A_J^{-1}b_J) > b_i$. 
In fact, even $A_i(A_J^{-1}b_J) \geq b_i + \frac{1}{\beta}$. 
But since we picked $\delta$ small enough,  $|A_ix^* - A_{i}(A_J^{-1}b_J)| = |A_iA_J^{-1}\delta \mathbf{1}| \leq n^2\cdot \Delta \cdot (n\Delta)^{n} \delta < \frac{1}{(n\Delta)^{n}} \leq \frac{1}{\beta}$, 
which is a contradiction.

\begin{figure}
\begin{center}
\psset{unit=2cm}
\begin{pspicture}(0.2,-1)(1.5,1.5)
%  \placeIIID{0}{0}{0}{\pnode(0,0){x00}}
%  \placeIIID{1}{0}{0}{\pnode(0,0){x10}}
%  \placeIIID{0}{1}{0}{\pnode(0,0){x01}}
%  \placeIIID{1}{1}{0}{\pnode(0,0){x11}}
%  \placeIIID{1.7}{0}{0}{\pnode(0,0){axesX}}
%  \placeIIID{0}{2}{0}{\pnode(0,0){axesY}}
%  \placeIIID{0}{0}{2}{\pnode(0,0){axesZ}}
%  \ncline[linewidth=0.75pt]{->}{x00}{axesX}
%  \ncline[linewidth=0.75pt]{->}{x00}{axesY}
%  \ncline[linewidth=0.75pt]{->}{x00}{axesZ}
  \pspolygon[fillstyle=solid,fillcolor=lightgray,linewidth=0.75pt](-0.3,-0.3)(1.7,-0.3)(-0.3,1.7)
  \pspolygon[fillstyle=solid,fillcolor=gray,linewidth=0.75pt](0,0)(1,0)(0,1)  \rput[c](0.3,0.3){$P$}
%  \psdots(0.4)(0.)(x11) 
  \psline[linestyle=dashed](x00)(x01)(x11)
  \pnode(0.5,0){L1A} \pnode(0.5,-0.3){L1B} \ncline{<->}{L1A}{L1B} \naput[labelsep=3pt]{$\delta$}
  \pnode(0,0.5){L2A} \pnode(-0.3,0.5){L2B} \ncline{<->}{L2A}{L2B} \nbput[labelsep=3pt]{$\delta$}
  \pnode(0.5,0.5){L3A} \pnode(0.7,0.7){L3B} \ncline{<->}{L3A}{L3B} \nbput[labelsep=3pt]{$\delta$}
  \psline[linewidth=1.5pt](-0.7,-0.3)(2.1,-0.3)
  \psline[linewidth=1.5pt](-0.7,2.1)(2.1,-0.7)
  \cnode*(1.7,-0.3){3pt}{x} \pnode(1.3,-0.7){xL} \nput{-135}{xL}{$x^* = A_J^{-1}(b_J + \delta\cdot\mathbf{1})$} \ncline{->}{xL}{x}
  \cnode*(1,0){3pt}{xII} \pnode(1.5,0.5){xIIL} \nput{45}{xIIL}{$A_J^{-1}b_J$} \ncline{->}{xIIL}{xII}
  \pnode(0.5,1.5){L4A} \pnode(-0.1,1.2){L4B} \ncline[arrowsize=5pt]{->}{L4A}{L4B} \nput{45}{L4A}{$\{ x \mid Ax \leq b + \delta \cdot \mathbf{1}\}$}
  \pnode(2.5,-0.5){L5A} \pnode(2.1,-0.3){L5B} \pnode(2.1,-0.7){L5C} \ncline[nodesepB=3pt,nodesepA=2pt]{->}{L5A}{L5B} \ncline[nodesepB=3pt,nodesepA=2pt]{->}{L5A}{L5C} \nput{0}{L5A}{$\in J$}
  \ncline[linestyle=dashed]{<->}{x}{xII}
\end{pspicture}
\caption{We bound the distance of $x^*$ to $P$ by the distance to $A_J^{-1}b_J$ (see dashed line).\label{fig:DistanceBound}}
\end{center}
\end{figure}
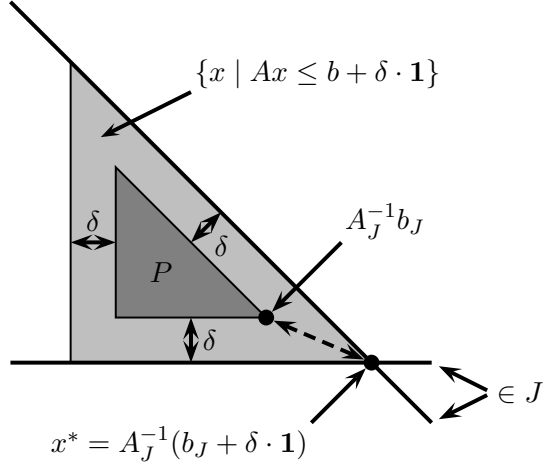

Hence we may assume that $J$ is indeed a feasible basis for $P$ and we can bound the distance of $x^*$ to $P$ by
the distance that the basic solution corresponding to basis $J$ ``moved'' by shifting the hyperplanes
by $\delta$ (see Figure~\ref{fig:DistanceBound}):
\[
  \|x^* - A_J^{-1}b_J\|_2
= \| A_J^{-1}(b_j+\delta\mathbf{1}) - A_J^{-1}b_J\|_2 = \|A_j^{-1}\delta\mathbf{1}\|_{2}
% \leq \delta\|A_J^{-1}\|_{\infty} 
\leq n\cdot \delta \cdot (\Delta n)^{n} \leq \varepsilon.
\]
Here we again used our choice of $\delta$.
\end{proofofclaim} 

Combining the proven claims yields  $P \subseteq \proj_x(Q) \subseteq P + \varepsilon$. 
% Both claims together yield that $\proj_x(Q) \subseteq P + \varepsilon$. 
% The ``furthermore'' follows from . 
\end{proof}

\section{Complexity theory considerations}

% In the following, let $G = (V,E)$ be the complete, undirected graph on $n$
% vertices. A \emph{Hamiltonian cycle} is a conneted subset $E' \subseteq E$
% of edges such that every node has degree $2$.

% The first intuitive answer would be that
% of course one cannot hope for a compact formulation for $P_{TSP}$, since
% otherwise one can run any polynomial time LP solver and hence derive 
% that $\mathbf{P} = \mathbf{NP}$. However, this answer is slightly incorrect
% to two reasons. Firstly, the pure existence of a compact extension does not
% necessarily imply that it can be generated for all $n$ by one algorithm.
% In other words, we have to consider \emph{non-uniform} algorithms.
% Secondly, at this point we cannot exclude that there are polytopes
% that only have compact formulations containing real numbers. 

% \begin{theorem}
% Unless $\mathbf{NP} \subseteq \mathbf{P/poly}$, there is no compact formulation
% for $P_{TSP}$.
% \end{theorem}

% The question remains, whether a compact formulation can always be realized 
% with rational numbers of polynomial encoding length. More formally:

% An affirmative answer to this question would be a strictly stronger 
% result than Theorem~?.

The set of problems that admit compact formulations induce a non-uniform complexity class in a natural way. 
In the following, we want to briefly discuss, how this class relates to other, well
studied classes. For an up-to-date introduction into the topic of complexity theory, 
we recommend the textbook of \cite{ComputationalComplexity-AroraBarak2009}.
Recall that $\{0,1\}^{*} = \bigcup_{n \geq 0} \{ 0,1\}^n$ is the set of all binary strings. 
By a slight abuse of notation we consider a $0/1$ string of length $n$ 
also as a binary vector of dimension $n$.
%We define $\mathbf{CF}$ as the set of languages $L \subseteq \{ 0,1\}^*$ that do have
%a compact formulation. % Let $\mathbf{CF}^{\textrm{enc}}$ be 

\begin{definition}
Let $\mathbf{CF}$ be the set of languages $L \subseteq \{ 0,1\}^*$ for which there exists a polynomial $p$
such that for all $n \in \setN$ there exist
 $A \in \setR^{p(n) × n}, B \in \setR^{p(n) × p(n)}, b \in \setR^{p(n)}$
such that
\[
  \conv(\{ x \in L : |x| = n\}) = \{ x \in \setR^n \mid \exists y \in \setR^{p(n)} : Ax + By \leq b \}.
\]
By $\mathbf{CF}^{\textrm{enc}} \subseteq \mathbf{CF}$ we denote the subclass of languages, for which
there exist integral matrices $A,B$ and vectors $b$ such that
%$A,B$ and $b$ can be chosen to be integral 
 $\log( \max\{ \|A\|_{\infty}, \|B\|_{\infty}, \|b\|_{\infty}\}) \leq p(n)$.
\end{definition}

%It is rather obvious that if a problem admits a compact formulation of polynomial encoding length, then
%it can be solved
Since any LP of polynomial size and encoding length can be solved in polynomial time, 
it is rather obvious that $\CFenc \subseteq \mathbf{P_{/poly}}$ (see also the remark of Yannakakis~\cite{ExpressingCOproblemsAsLPs-Yannakakis1991}).
However, Theorem~\ref{thm:DescriptionOfX} also provides a slightly stronger claim:
\begin{theorem}
$\mathbf{CF} \subseteq \mathbf{P_{/poly}}$.
\end{theorem}
\begin{proof}
Let $L \in \mathbf{CF}$ and $X = L \cap \{ 0,1\}^n$ for some $n\in \setN$ and let $r := \xc(\conv(X))$.
Recall that $r$ must be polynomial in $n$. 
It suffices to provide a 
Turing machine that takes polynomial advice (see~\cite{ComputationalComplexity-AroraBarak2009}).
Our advice for all input strings $x$ of length $n$ consists in the matrices $\bar{A},\bar{U},\bar{b}$ 
provided by Theorem~\ref{thm:DescriptionOfX}.
Note that their encoding length is bounded by a polynomial in $n$ and $r$. To verify whether $x \in X$, 
we simply test whether the following polynomial size linear system has a solution $y$:
\begin{eqnarray*}
  - \frac{1}{4r(n+r)} \leq \bar{A}x + \bar{U}y - \bar{b} &\leq& \frac{1}{4r(n+r)} \\
  0 \leq y_j &\leq& \Delta \quad \forall j=1,\ldots,r
\end{eqnarray*}
This can be done in polynomial time~\cite{Khachiyan79}.
\end{proof}

%One question, that we do not know the answer so far is: 
%Given a compact formulation for a $0/1$ polytope, can one always find an equivalent
% polynomial size formulation in which the numbers have polynomial encoding length? 
%In short

% Is it true that for any $0/1$ polytope $P$ for which $\rk_+(P)$ is bounded by a polynomial, 
% one can always find a polynomial size extension with polynomial encoding length 

% Does there exists a polynomial $p(n,r)$ such that 
% for any subset $X \subseteq \{ 0,1\}^n$ there exists an extension
% \[
%   Q = \{ x \mid Rx + Uy = c, b \geq \mathbf{0} \}
% \] 
% (i.e. $\conv(X) = \{ x \mid \exists y \geq \mathbf{0}: Rx + Uy = c$)
% such that $U$ has at most $poly(n,\rk_+(X))$ columns and $R,U,c$ have rational entries
% of encoding length $p(n,\rk_+(X))$.
We make the following conjecture:
\begin{conjecture}
$\CFenc = \CF$.
\end{conjecture}

One of the most popular polytopes in the literature is the \emph{TSP polytope}~(see e.g. \cite{ExpressingCOproblemsAsLPs-Yannakakis1991,All01polytopesAppearAsFacesOfTheTSPpolytope1996}), hence we want
to discuss how it relates to the class $\CF$.
Let $K_n$ be the complete undirected graph on $n$ nodes. We define a language
\[
\texttt{TSP} = \bigcup_{n \in \setN} \{ \chi(C) \in \setR^{n \choose 2} \mid C\subseteq E_n \textrm{ is Hamiltonian cycle in } K_n=([n],E_n) \}
\]
(here $\chi(C)$ denotes the characteristic vector of $C$).
%Then we obtain a language by letting $\texttt{TSP} := \bigcup_{n \in \setN} \texttt{TSP}_n$.
Again it is obvious that $\mathbf{NP} \not\subseteq \mathbf{P_{/poly}} \Rightarrow \texttt{TSP} \notin \CFenc$, but also here we
can show a slightly stronger claim:
\begin{theorem} \label{thm:NoCFforTSP}
$\mathbf{NP} \not\subseteq \mathbf{P_{/poly}} \Rightarrow \texttt{TSP} \notin \CF$.
In other words, unless $\mathbf{NP}$ problems do not all have polynomial size circuits,
 the TSP polytope does not have a compact formulation, even if arbitrary real numbers are allowed.
\end{theorem}
\begin{proof}
Suppose for the sake of contradiction that $\texttt{TSP} \in \CF$. 
By $\mathbf{NP}$-hardness of the \emph{Hamiltonian Cycle problem}~\cite{GareyJohnson79}, 
given a cost vector $c \in \{ 1,2\}^{{n \choose 2}}$ it
is $\mathbf{NP}$-hard to decide, whether there is an $x \in \texttt{TSP}$ with $c^Tx \leq n$. 
Consider the Turing machine (taking polynomial advice), which optimizes $c$ over
the polytope $Q$ from Theorem~\ref{thm:Approximation-P-by-Q} for $\varepsilon := \frac{1}{2n}$ 
and let $x^*$ be an optimum fractional solution. 
% We claim that $c^Tx^* \leq \Leftrightarrow c^Tx \leq$
If there is an $x \in \texttt{TSP}$ with $c^Tx \leq n$, then $c^Tx^* \leq n$.
Otherwise, $c^Tx^* \geq (n+1) - \varepsilon \|c\|_2 > n$. Hence the Turing machine
decides an $\mathbf{NP}$-hard problem, which implies the claim.
\end{proof}
Note that $\texttt{TSP} \in \mathbf{P_{/poly}}$, since \emph{testing} whether $x$ is
the characteristic vector of a Hamiltonian cycle
is easy. Just \emph{optimizing} over all those vectors is difficult. 

We should not introduce a new complexity class $\CF$, without relating it to already known ones. 
We saw already that $\CF \subseteq \mathbf{P_{/poly}}$, so what about other non-uniform complexity classes 
within $\mathbf{P_{/poly}}$? Certainly the most studied of those classes is $\mathbf{AC}^0$,
which is the set of languages for which there are circuits with bounded
depth and unbounded fan-in.

Recall that $\tt{PARITY}$ is the set of all $x \in \{ 0,1\}^{*}$ such $\| x\|_1$ is odd. 
Then $\texttt{PARITY}$ admits 
a compact formulation (with small integral coefficients; 
see e.g.~\cite{ExtendedFormulationsSurvey-ConfortiCornuejolsZambelli-2010}), thus $\texttt{PARITY} \in \mathbf{CF}^{\textrm{enc}}$. 
In a seminal result,  Furst, Saxe and Sipser~\cite{ParityNotInAC0-FurstSaxeSipser84} showed that $\texttt{PARITY} \notin \mathbf{AC}^0$ and 
hence $\CF \not\subseteq \mathbf{AC}^0$ (in fact, even $\mathbf{CF}^{\textrm{enc}} \not\subseteq \mathbf{AC}^0 $).
%This allows us a separation from a well studied class of very restricted circuits, namely $\mathbf{AC}^0$.
%Recall that  
On the other hand, under widely believed assumptions also the reverse is true:
\begin{theorem}
$\mathbf{NP} \not\subseteq \mathbf{P_{/poly}} \Rightarrow \mathbf{AC}^0 \not\subseteq \mathbf{CF} $.
\end{theorem}
\begin{proof}
We need to exhibit a problem, which can be solved by constant depth circuits, but is likely not
to be in $\CF$. Consider the complete tripartite graph $G_n=([n]^3,E_n)$, i.e. for any distinct $i,j,k \in [n]$, one has a
triple $e=\{ i,j,k\} \in E_n$. We say that a subset $E' \subseteq E_n$ is a \emph{(3-dimensional) matching} if all triples
in $E'$ are disjoint. Define
\[
  \texttt{3DM} = \bigcup_{n \geq 1} \{ \chi(E') \mid E' \subseteq E_n \textrm{ is matching} \}
\]
%and $\texttt{3DMatching} := \bigcup_{n \geq 1} \texttt{3DMatching}$ as the corresponding language. 
Given a cost vector $c \in \{ 0,1\}^{E_n}$, it is  $\mathbf{NP}$-hard to decide, whether there is an $x \in \texttt{3DM}$ with $c^Tx = n$~\cite{GareyJohnson79}
(i.e. whether there is a perfect 3-dimensional matching contained in $\{ e \in E \mid c_e = 1\}$).
Within the same line of arguments as in Theorem~\ref{thm:NoCFforTSP} one has $\texttt{3DM} \notin \CF$ unless $\mathbf{NP} \subseteq \mathbf{P_{/poly}}$. 
% It remains to prove that $\texttt{3DMatching}$ lies in $\mathbf{AC}^0$. But one can easily see that 
Finally it is not difficult to see that 
\[
 \bigwedge_{e, e' \in E: 1\leq |e\cap e'|\leq 2}  (¬ x_e \lor ¬ x_{e'})
\]
is a polynomial size, constant depth formula for $\texttt{3DM}$, thus $\texttt{3DM} \in \mathbf{AC}^0$.
\end{proof}

\paragraph{Acknowledgements.}

The author is grateful to Samuel Fiorini for carefully reading a preliminary 
draft. 
Furthermore the author wants to thank Michel X. Goemans, Neil Olver and Rico Zenklusen for helpful comments.

\bibliographystyle{alpha}
\bibliography{sizeOfExtendedFormulations}

\end{document}